\documentclass[12pt,reqno]{amsart}
\usepackage{tikz}
\usetikzlibrary{arrows.meta, positioning}
\usepackage[mathscr]{eucal}
\usepackage{caption}
\usepackage{subcaption}
\usepackage{amscd}
\usepackage{cite}
\usepackage{amsfonts}
\usepackage{enumerate}
\usepackage{amsmath, amsthm, amssymb}
\usepackage{latexsym}
\usepackage{tikz-cd}
\usepackage{graphicx}
\usepackage{epstopdf}
\usepackage{hyperref}
\newtheorem{theorem}{Theorem}[section]

\newtheorem{proposition}{Proposition}[section]
\newtheorem{corollary}{Corollary}[section]
\theoremstyle{definition}
\newtheorem{definition}{Definition}[section]
\newtheorem{example}{Example}[section]
\newtheorem{observation}{Observation}[section]

\theoremstyle{remark}
\newtheorem{remark}{Remark}[section]
\numberwithin{equation}{section}

\begin{document}
	\title[Polynomial $\mathcal{Z}$-contraction]{Polynomial $\mathcal{Z}$-contraction: A unified approach with Application}
	\author[Hasanuzzaman]{Md Hasanuzzaman$^{1}$, Stojan N. Radenovi{\'c}$^{2}$}
	%\footnote{Corresponding author: Md Hasanuzzaman}
	\maketitle
	\begin{center}
		{\footnotesize $^{1}$Department of Mathematics, Thapar Institute of Engineering \& Technology, Patiala-147004, India,\\
		$^{2}$Faculty of Mechanical Engineering, University of Belgrade, Kraljice Marije 16, Beograd 35, 11120, Serbia}

		{\footnotesize $^{1}$md.hasanuzzaman1@gmail.com, $^{2}$radens@beotel.net}
	\end{center}
	\begin{abstract} In this paper, we introduce polynomial $\mathcal{Z}$-contraction by unifying the ideas of polynomial contraction and  $\mathcal{Z}$-contraction and also provide demonstrative examples to show the effectiveness and novelty. Chiefly, we prove two fixed point results for polynomial $\mathcal{Z}$-contraction employing simulation functions. Lastly, we utilize our main result to discuss the existence and uniqueness of solution for  boundary value problem with a concrete example in $L^p(0,\pi)$ under suitable assumptions.
	\end{abstract}
	
	
\section{Introduction} 

Metric fixed point theory laid its foundation on the pioneer work of Banach \cite{banach1922operations} on contraction mapping which asserts that, every contraction on a complete metric space has a unique fixed point. This result widely known as Banach's contraction principle (in short BCP) in the literature.  Due to the simplicity, utility and numerous applications in various fields within and beyond mathematics of this classical result, the extension and generalization of BCP is old but still an active area of research till the date.

In 1962, Edelstein \cite{edelstein1962fixed} obtained a fixed point theorem for contractive(sometimes called, strict contraction) mapping on compact metric space. Thereafter, out of many, one of the main goals in metric fixed point theory remains to obtain existence and uniqueness fixed point results for those class of mappings that lies in between contraction and contractive classes in complete metric space. For some notable development in this direction, one can refer \cite{wardowski2012fixed,khojasteh2015new,matkowski1977fixed,boyd1969nonlinear,jachymski2008contraction,rhoades1977comparison} and references cited therein.

In 2015, with a view to enlarge the class of contraction and to obtain a fixed point result in complete metric space which also in turn unifies several existing classical results (linear as well as nonlinear contractions), Khojasteh et al.\cite{khojasteh2015new} proposed the idea of $\mathcal{Z}$-contraction using a class of control functions called as simulation functions and proved a fixed point result for the same. As every $\mathcal{Z}$-contraction is contractive (see Remark \eqref{12p1r}), therefore the class of $\mathcal{Z}$-contraction is contained in class of  contractive mapping and hence Khojasteh et al.\cite{khojasteh2015new} have successfully obtained fixed point results for a mapping that lies  in between contraction and contractive classes of mappings acting on a complete metric space. For some related results and their applications in this direction, one can consult, e.g., \cite{Hasan2020,jain2023,jain2025a, rad2018} and references cited therein.

Very recently, Jleli et al. \cite{Jleli2025} introduced polynomial contraction  and prove fixed point results for such contraction that remains yet another substantial generalization of Banach contraction principle in a slightly different and meaningful way. To be precise, authors \cite{Jleli2025} proved two fixed point results  for polynomial contraction, firstly for continuous mapping and later for Picard continuous mapping(not necessarily continuous). This approach replaces the simple distance function in the contraction inequality with a polynomial expression of distances, allowing higher-order polynomial distance terms with point-dependent weights.

Motivated by above, in this paper, we introduce the notion of polynomial $\mathcal{Z}$-contraction, and provide some examples to exhibit the utility and effectiveness of our newly introduced contraction over existing similar results. We prove two fixed point results for polynomial $\mathcal{Z}$-contraction in complete metric space. This new class of mappings unifies the implicit form of $\mathcal{Z}$-contractions(containing linear as well as nonlinear contractions) with the algebraic flexibility of underlying metric in polynomial contractions.  Our results not only generalize the main results of Khojasteh et al.\cite{khojasteh2015new} and Jleli et al. \cite{Jleli2025} but also demonstrate the underlying principles that can be combined to yield more comprehensive and unified theory in this direction.\\

In what follows, the paper is organized as: Section 2 is devoted to review and recollect the necessary definitions, facts and auxiliary results related to $\mathcal{Z}$-contraction and polynomial contraction that are needed in our subsequent discussions. Section 3 deals with main results on polynomial $\mathcal{Z}$-contraction, initially for continuous mappings and later for Picard-continuous mappings. Thereafter, we derive a few corollaries and provide  examples to illustrate the scope and significance of main results. Lastly, Section 4 contains an application to the existence and uniqueness of solution for semilinear elliptic Dirichlet  boundary value problem along with an example in $L^p(0,\pi)$ under suitable assumption.\\

Let us summarize main contributions of this paper as follows:
\begin{itemize}
	\item We introduce a class of polynomial $Z$-contractions in complete metric spaces, wherein the contractive condition is formulated through a polynomial expression of distances involving a simulation function.
	
	\item The proposed class admits mappings that do not satisfy any known $Z$-contraction condition and are not covered by existing polynomial contraction condition, as illustrated by explicit example (see Example \ref{exm1} ).
	
	\item Two fixed point theorems are established for polynomial $Z$-contractions under standard completeness and continuity assumptions, including the case of Picard-continuous mappings. A separate example for Picard-continuous mappings is also presented (see Example \ref{exm3}).
	
	\item As an application, the abstract results are used to study the existence and uniqueness of solutions of a nonlinear boundary value problem in $L^{p}(0,\pi)$ with a demonstrative example.
\end{itemize}

\section{Preliminaries}
We shall reminisce the following terminological and notational conventions to make our discussion possibly self-sustained within this manuscript. In what follows, $\mathbb{N}$ and $\mathbb{R}$ respectively denote the sets of natural numbers, real numbers and $\mathbb{N}_0=\mathbb{N}\cup\{0\}$. Also, $Fix(S)$ is the set of all fixed point of a mapping $S,\ i.e., \ Fix(S)=\{r| S(r)=r\}.$\\

Firstly, we recall the class of simulation functions introduced by Khojasteh et al.\cite{khojasteh2015new}.
\begin{definition}\label{2p1t1d}
	A mapping $\zeta:[0,+\infty)\times[0,+\infty)\rightarrow\mathbb{R}$ is called a simulation function if the following are satisfied:
	\begin{enumerate}[($\zeta_1$)]
		\item $\zeta(0,0)=0$;
		\item $\zeta(t,s)<s-t$ for all $t,s>0$;
		\item if $\{t_n\}, \{s_n\}$ are sequences in $(0,+\infty)$ with $\displaystyle\lim_{n\to+\infty}t_n=\displaystyle\lim_{n\to+\infty}s_n>0$, then $ \ \ \ \ \ $ $\displaystyle\limsup_{n\to+\infty}\zeta(t_n,s_n)<0.$
	\end{enumerate}
\end{definition}  
Here onward, we shall denote the family of all simulation functions by $\mathcal{Z}$.

\begin{example}
	``Let $\zeta_i:[0,+\infty)\times[0,+\infty)\rightarrow\mathbb{R}$ for $i=1,2,3,4,5$ be defined as follows:
	\begin{enumerate}[$\bullet$]
		
		\item $\zeta_1(t,s)=\alpha s-t$ for all $t,s\in[0,+\infty)$, and $\alpha\in[0,1)$.
		
		\item $\zeta_2(t,s)=\Psi(s)-\Phi(t)$ for all $t,s\in[0,+\infty)$, where $\Phi,\Psi:[0,+\infty)\rightarrow[0,+\infty)$ are two continuous functions such that $\Psi(t)=\Phi(t)=0$ if and only if $t=0$ and $\Psi(t)<t\leq\Phi(t)$ for all $t>0$.
		
		\item $\zeta_3(t,s)=s-\eta(s)-t$ for all$t,s\in[0,+\infty)$, where $\eta:[0,+\infty)\rightarrow[0,+\infty)$ is a lower semi continuous function and $\eta(t)=0$ if and only if $t=0$.
		
		\item $\zeta_4(t,s)=s-\int_{0}^{t}\varphi(t)dt$ for all $t,s\in[0,+\infty)$, where $\varphi:[0,+\infty)\rightarrow[0,+\infty)$ is a function such that $\int_{0}^{\epsilon}\varphi(t)dt$ exists and $\int_{0}^{\epsilon}\varphi(t)dt>\epsilon$, $\epsilon>0.$ 

		\item $\zeta_5(t,s)=\frac{s}{\alpha s+1}-$ for all $t,s\in[0,+\infty)$ and $\alpha>0$.
	\end{enumerate}
	Then $\zeta_i\in \mathcal{Z}$ for all $i=1,2,3,4,5$."
\end{example}
For further examples and  results related to simulation functions, one may refer  \cite{khojasteh2015new,roldan2015coincidence,karapinara2016fixed,karapinar2017approach,alsulami2014proposal} and reference cited therein. For a detailed survey on simulation functions, one may consult \cite{Chanda2019}. \\
Now, we recall the definition of $\mathcal{Z}$-contraction introduced by Khojasteh et al.\cite{khojasteh2015new}.
\begin{definition}\cite{khojasteh2015new}\label{khojasteh}
	Let $(M,d)$ be a metric space. Then $S:M\rightarrow M$ is said to be a $\mathcal{Z}$-contraction for $\zeta\in\mathcal{Z}$ if
	\begin{equation}
	\zeta(d(Sr,Ss),d(r,s))\geq0\ \ \text{for all}\ r,s\in M. 
	\end{equation} 
\end{definition}

\begin{remark}\label{12p1r}
	We want to emphasize two important fact about $\mathcal{Z}$-contraction due to the condition $(\zeta_2)$ of Definition \ref{2p1t1d}.
	\begin{itemize}
		\item[$\bullet$] [cf. \cite{Hasan2020}] If a mapping $S$ is a $\mathcal{Z}$-contraction then it cannot be an isometry as $\zeta(x,x)<0,$ for all $x>0$.
		\item[$\bullet$] [cf. \cite{khojasteh2015new}]  It is also easy to observe that for all $x\geq y>0,~ \zeta(x,y)<0$, so if $S$ is a $\mathcal{Z}$-contraction with respect to $\zeta\in \mathcal{Z}$, then  $d(Sr,Ss)<d(r, s)$ whenever $r\neq s$ for all $r,s\in M$. Hence, every $\mathcal{Z}$-contraction is contractive and therefore continuous.	
	\end{itemize}
	If $\zeta(x,y)=\lambda y-x$ for all $x,y\in[0,+\infty)$ with $\lambda\in[0,1),$ then $\mathcal{Z}$-contraction turn to classical Banach contraction. Thus, the class of $\mathcal{Z}$-contractions essentially lies between the class of classical contraction and contractive mappings. Many attempt have been made earlier to obtain fixed point results for those class of mappings that lies in between contraction and contractive class of mappings in complete metric space, the class of $\mathcal{Z}$-contractions is one of them.
\end{remark}
 
Next, we recall the result of Khojasteh et al.\cite{khojasteh2015new} for $\mathcal{Z}$-contraction in a complete metric space that utilizes the properties of simulation functions. 
\begin{theorem}\cite{khojasteh2015new}\label{khojastehth}
	Let $(M,d)$ be a complete metric space and $S:M\rightarrow M$ be a $\mathcal{Z}$-contraction for some $\zeta$. Then $S$ has a unique fixed point.
\end{theorem}

Very recently, Jleli et al. \cite{Jleli2025} introduced the idea of polynomial contraction and extended Banach contraction principle in a slightly different and meaningful way that mainly replaces the simple distance function in the contraction inequality with a polynomial expression of distances.

\begin{definition}\label{jlelipoly2025def}
	Let $(M,d)$ be a metric space and $S:M\rightarrow M$. Then $S$ is said to be polynomial contraction if there exist $\lambda\in [0,1), ~k\in \mathbb{N}$  and mappings $\alpha_i:M\times M\to [0,+\infty)$ for $i=0,1,\ldots,k$ such that
	\begin{equation}\label{jlelipoly2025}
		\sum_{i=0}^{k}\alpha_i(Sr,Ss)d^{i}(Sr,Ss)\leq\lambda \sum_{i=0}^{k}\alpha_i(r,s)d^{i}(r,s), \ \text{for all}\ r,s\in M.
	\end{equation} 
\end{definition}
Let us denote $\displaystyle\mathbb{P}(r,s):=\sum_{i=0}^{k}\alpha_i(r,s)d^{i}(r,s),$ then above contraction condition \eqref{jlelipoly2025} can be viewed as 
\begin{equation*}
	\mathbb{P}(Sr,Ss)\leq\lambda~\mathbb{P}(r,s),~ \lambda\in [0,1).
\end{equation*}

\begin{theorem}\cite{Jleli2025}\label{Jlelipolyth}
	Let $(M,d)$ be a complete metric space and let $S:M\to M$ be a mapping.
	Assume that there exist a constant $\lambda\in[0,1)$, a natural number
	$k\ge1$, and a family of mappings $\alpha_i:M\times M\to[0,+\infty), \ i=0,1,\ldots,k,$ such that for all $r,s\in M$,
	\begin{equation}\label{poly_contr_JPS}
		\mathbb{P}(Sr,Ss)\leq\lambda~\mathbb{P}(r,s).
	\end{equation}
	Further, assume  that if each $\alpha_i$ is continuous and there exists $j\in\{1,\ldots,k\}$ and a constant $A_j>0$ such that $\alpha_j(r,s)\ge A_j\ \text{for all } r,s\in M$. Then $S$ has a unique fixed point $r^*\in M$, and for every $r\in M$,
	the Picard sequence $\{S^n r\}_{n\in\mathbb N}$ converges to $r^*$.
\end{theorem}

\begin{definition}
	Let $(M,d)$ be a metric space, $S:M\rightarrow M$. Then $S$ is said to be Picard-continuous if for any $t\in M$,
	\begin{center}
		$\displaystyle\lim_{n\to +\infty}d(S^n{t_0},t)=0$~ for some $t_0\in M$ implies $\displaystyle\lim_{n\to +\infty}d(S(S^n{t_0}),S(t))=0$.
	\end{center}
	That is, $S$ is continuous along its own Picard iterates.	
\end{definition}
\begin{remark}
	A few facts regarding  Picard-continuous mappings.
	\begin{itemize}
		\item Every continuous mapping is Picard continuous, but the converse need not be true. Picard continuity only requires continuity along sequences generated by iteration of $S$, not along all convergent sequences. For example, consider $M=[0,1]\subset\mathbb{R}$ and define $S:M\to M$ by 
		\begin{equation*}
			S(r)=\left\{
			\begin{array}{ll}
				0, & \text{for} \ \hbox{$r\in[0,1)$}\\
				\frac{1}{k}, & \text{for}~ r=1 ~\ \hbox{(where, $k$ is a fixed natural number.)}
			\end{array}
			\right.
		\end{equation*}
		Then, it is easy to check that $S$ is not continuous (at $r=1$) but it is Picard continuous.

		\item  Sometimes such mappings are called orbitally continuous mappings, since continuity is required only along the orbit $\{r,Sr,S^2r,\cdots\}$
	\end{itemize}
In the next section, we show that our main result ($i.e.,$ Theorem \ref{4p1th}) is expandable for Picard-continuous mappings as well.
\end{remark}

\section{Main Results}
In this section, firstly we define polynomial $\mathcal{Z}$-contraction and make some observations and consequences. Then, we obtain fixed point results for mappings that satisfy polynomial $\mathcal{Z}$-contraction condition. Our approach will be similar to Khojasteh et al.\cite{khojasteh2015new} in the light of polynomial contraction introduced by Jleli et al. \cite{Jleli2025}, however the obtained result will be more generalized in nature.  

\begin{definition}
	Let $(M,d)$ be a metric space and $S:M\rightarrow M$. Then $S$ is said to be polynomial $\mathcal{Z}$-contraction for some $\zeta\in\mathcal{Z}$, if there exists positive integer $k$, and mappings $\alpha_i:M\times M\to [0,+\infty)$ for $i=0,1,\ldots,k$ such that
	\begin{equation}
		\zeta(\mathbb{P}(Sr,Ss),\mathbb{P}(r,s))\geq0\ \ \ \text{for all}\ r,s\in M,
	\end{equation} 
	where $\mathbb{P}(r,s):=\sum_{i=0}^{k}\alpha_i(r,s)d^{i}(r,s).$
\end{definition}

\begin{remark}
	From above, it is clear that the polynomial $\mathcal{Z}$-contraction is of course a proper generalization as well as unification of existing standard contraction conditions, because it contains both $\mathcal{Z}$-contractions and standard polynomial contractions as special cases.
	\begin{enumerate}
		\item If we set $k=1, \alpha_0=0, and~ \alpha_1\equiv 1$ (a constant function with value 1), the above definition reduces
		to the $\mathcal{Z}$-contraction (see Definition \ref{khojasteh}).
		\item If we choose $\zeta(t,s)=\lambda s-t$ for all $t,s\in[0,+\infty)$ with $\lambda\in[0,1),$ then it reduces to the standard polynomial
		contraction. 
	\end{enumerate}
\end{remark}

The following implications naturally holds. 
$$\boxed{
	\begin{array}{ccc}
		\textbf{Banach contraction} & \longrightarrow & \boldsymbol{\mathcal{Z}\textbf{-contraction}} \\[6pt]
		\Big\downarrow &  & \Big\downarrow \\[6pt]
		\textbf{Polynomial contraction} & \longrightarrow & \boldsymbol{\textbf{Polynomial } \mathcal{Z}\textbf{-contraction}}
	\end{array}
}$$

Reverse implications need not always true in general~(see Example \ref{exm1}).

\begin{remark}
	Since polynomial contractions need not necessarily continuous (cf. \cite{Jleli2025}) therefore unlike $\mathcal{Z}$-contractions, polynomial $\mathcal{Z}$-contractions are not necessarily continuous.
\end{remark}
However, the following proposition gives the necessary conditions under which polynomial $\mathcal{Z}$-contractions are continuous.
\begin{proposition}\label{prop1}
Let $(M,d)$ be a metric space, and $S:M\to M$ be a  polynomial $\mathcal{Z}$-contraction for some $\zeta\in\mathcal{Z}$. Assume that the coefficient functions $\alpha_i(r,s)$ are continuous and there exist $j\in \{1,\ldots, k\}$ and a constant $A_j>0$ such that $\alpha_j(r,s)\geq A_j$ for all $r, s\in M$. Furthermore, assume that $\alpha_0(r,s)=0,  r,s\in M$. Then S is a continuous.
\begin{proof}
Let $\{r_n\}\subset M$ be a sequence such that $r_n\to r$ for some $r\in M$. 
As $\alpha_i$ and $d^i$ are continuous, $\mathbb{P}(r,s)=\sum_{i=0}^{k}\alpha_i(r,s)d^{i}(r,s)$ is also continuous in both variables. Thus,
$$ \lim_{n\to +\infty}\mathbb{P}(r_n,r)=\mathbb{P}(r,r)=0.$$
Now, if $Sr_n=Sr$ for all sufficiently large $n\in\mathbb{N}$, then the result is obvious. Otherwise, assume that $\mathbb{P}(Sr_n,Sr)>0$, then by polynomial $\mathcal{Z}$-contraction condition we get 
\begin{eqnarray*}
	0\leq\zeta(\mathbb{P}(Sr_n,Sr), \mathbb{P}(r_n,r))< \mathbb{P}(r_n,r)-\mathbb{P}(Sr_n,Sr)
\end{eqnarray*}
  implies that $\mathbb{P}(Sr_n,Sr)< \mathbb{P}(r_n,r)$ and by taking limit we get $\displaystyle\lim_{n\to +\infty}\mathbb{P}(Sr_n,Sr)=0$.

 Also, from the hypothesis,  $\alpha_j(r,s)\geq A_j>0$ gives the bound  
\begin{equation}
	\mathbb{P}(Sr_n,Sr)\geq A_jd^{j}(Sr_n,Sr)>0.
\end{equation}
Therefore, we must have $\displaystyle\lim_{n\to +\infty}	d^{j}(Sr_n,Sr)=0$ for any such $j$.

As, $\alpha_0(r,s)=0$, for all $r,s\in M$, we can conclude from above that  $\displaystyle\lim_{n\to +\infty}	d(Sr_n,Sr)=0$ and hence $S$ is continuous.
\end{proof}
\end{proposition}

Here, we present our first main result for a continuous self mapping $S$ on $M$ that enjoys polynomial $\mathcal{Z}$-contraction condition on a complete metric space.
\begin{theorem}\label{4p1th}
	Let $(M,d)$ be a complete metric space, and $S:M\to M$ be a continuous polynomial $\mathcal{Z}$-contraction for some $\zeta\in\mathcal{Z}$. Assume that the coefficient functions $\alpha_i(r,s)$ are uniformly continuous (with respect to the usual product metric) and there exist $j\in \{1,\ldots, k\}$ and a constant $A_j>0$ such that $\alpha_j(r,s)\geq A_j$ for all $r, s\in M$. Then $S$ has a unique fixed point $r^*\in M$ and for every $r\in M$ the sequence $\{S^n{r}\}_{n\in\mathbb{N}}$ converges to $r^*$(where $S^n:= S \underset{n-\text{times}}{\underbrace{\circ \cdots\circ}} S$).
\end{theorem}
\begin{proof}
Let $r_0$ be an arbitrary point in $M$. Construct a Picard sequence $\{r_n\}$  generated by $r_0$ as $r_n=Sr_{n-1}(=S^nr_0), n\in\mathbb{N}_0$.
Now, if there exists some $n_0\in\mathbb{N}_0$ such that $r_{n_0}=Sr_{n_0}$, then the result follows immediately.

Therefore, without loss of generality, let us assume that for all $n\in\mathbb{N}_0$, $r_n\not=r_{n+1}$   so that $d(r_{n},r_{n+1})>0$. Since $\mathbb{P}(r_n,r_{n+1}):=\sum_{i=0}^{k}\alpha_i(r_n,r_{n+1})d^{i}(r_n,r_{n+1}),$ we have $\mathbb{P}(r_n,r_{n+1})>0$. As $S$ is polynomial $\mathcal Z$-contraction, we have

	\begin{eqnarray}\label{4p1th3eq}
	\nonumber0&\leq&\zeta\big(\mathbb{P}(Sr_{n},Sr_{n+1}),\mathbb{P}(r_n,r_{n+1}))\\
	\nonumber&<&\mathbb{P}(r_n,r_{n+1})-\mathbb{P}(Sr_{n},Sr_{n+1})\\
    \mathbb{P}(Sr_{n},Sr_{n+1})&<&\mathbb{P}(r_n,r_{n+1}),
	\end{eqnarray}
	which implies that $\mathbb{P}(r_{n+1},r_{n+2})<\mathbb{P}(r_n,r_{n+1})$ for all $n\in\mathbb{N}_0$. Therefore, $\{\mathbb{P}(r_n,r_{n+1})\}_{n=0}^{+\infty}$ is a monotonically decreasing sequence of positive real numbers, and hence there exists $l\geq0$, such that $\displaystyle\lim_{n\to+\infty}\mathbb{P}(r_n,r_{n+1})=l.$ We claim that $l=0.$ On contrary, let us assume that $l>0$ then using $(\zeta_3)$, we obtain
	\begin{equation*}
	0\leq\limsup_{n\to+\infty}\zeta\big(\mathbb{P}(Sr_n,Sr_{n+1}),d(r_n,r_{n+1})\big)<0,
	\end{equation*}
	which is a contradiction and hence $l=0, i.e.,$
	\begin{equation}\label{p1th1eq}
	\lim_{n\to+\infty}\mathbb{P}(r_n,r_{n+1})=0.
	\end{equation} 
	Now, from the hypothesis of our theorem,  $\alpha_j(r,s)\geq A_j>0.$ Therefore, 
	\begin{eqnarray*}
		\nonumber \mathbb{P}(r_n,r_{n+1})&=&\sum_{i=0}^{k}\alpha_i(r_n,r_{n+1})d^{i}(r_n,r_{n+1})\\
		\nonumber&\geq&\alpha_j(r_n,r_{n+1})d^{j}(r_n,r_{n+1})\\
		&\geq&A_jd^{j}(r_n,r_{n+1})>0,
	\end{eqnarray*}
	Thus, from \eqref{p1th1eq} we must have $A_jd^{j}(r_n,r_{n+1})\to0$ (as $n\to+\infty$) and consequently
	\begin{equation}\label{12p1th1eq}
		\lim_{n\to+\infty}d(r_n,r_{n+1})=0.
	\end{equation}
	This shows that the sequence $\{r_n\}$ obtained above is asymptotically regular.
	
	We now prove that $\{r_n\}$ is Cauchy. Suppose, to the contrary, that $\{r_n\}$ is not Cauchy. Then there exists $\epsilon>0$ and $l_0\in\mathbb{N}_0$ with $m(l)>n(l)>l \geq l_0,$ such that
	\begin{equation*}
	d(r_{m(l)}, r_{n(l)}) \ge \epsilon, \qquad d(r_{m(l)}, r_{n(l)-1}) < \epsilon.
	\end{equation*}

	Using the triangle inequality and \eqref{12p1th1eq}, we obtain
	\begin{equation}\label{12p1th2eq0}
			\lim_{l\to+\infty} d(r_{m(l)}, r_{n(l)}) = \epsilon. 
			%tag3
	\end{equation}
	
	Also, since $d(r_{n+1}, r_n) \to 0$, another application of the triangle inequality gives
	\begin{equation}\label{12p1th3eq0}
		\lim_{l\to+\infty} d(r_{m(l)+1}, r_{n(l)+1}) = \epsilon.
		% \tag{4}
	\end{equation}
	Now, from \eqref{12p1th1eq} we also have
	$$
	d(r_{m(l)+1}, r_{m(l)}) \to 0, \qquad d(r_{n(l)+1}, r_{n(l)}) \to 0.
	$$
	Consider the product metric $d_\times$ on $M \times M$ defined by
	$$
	d_\times\big((x,y),(u,v)\big) := d(x,u) + d(y,v).
	$$
	Then
	\begin{equation}\label{12p1th4eq}
		d_\times\big((r_{m(l)+1}, r_{n(l)+1}), (r_{m(l)}, r_{n(l)})\big)
	= d(r_{m(l)+1}, r_{m(l)}) + d(r_{n(l)+1}, r_{n(l)}) \to 0. 
	%\tag{5}
	\end{equation}
	
	Since each $\alpha_i$ is uniformly continuous on $M \times M$ with respect to $d_\times$, and \eqref{12p1th4eq} holds, we get
	\[
	\lim_{l\to+\infty} \big| \alpha_i(r_{m(l)+1}, r_{n(l)+1}) - \alpha_i(r_{m(l)}, r_{n(l)}) \big| = 0 \quad \text{for each } i.
	\]
	Combining this with \eqref{12p1th2eq0} and \eqref{12p1th3eq0}, and using the fact that $\mathbb{P}(r,s) = \sum_{i=0}^k \alpha_i(r,s) d^i(r,s)$, we conclude
	
	\begin{eqnarray}\label{12p1th5eq}
		\lim_{l\to+\infty} \mathbb{P}(r_{m(l)+1}, r_{n(l)+1})
		&=
		\lim_{l\to+\infty} \mathbb{P}(r_{m(l)}, r_{n(l)}) &=: L_0. 
		%\tag{6}
	\end{eqnarray}
	
	Moreover, since $\alpha_j \ge A_j > 0$ and $d(r_{m(l)}, r_{n(l)}) \to \epsilon > 0$, we have
	$$
	\mathbb{P}(r_{m(l)}, r_{n(l)}) \ge A_j d^j(r_{m(l)}, r_{n(l)}) \to A_j \epsilon^j > 0,
	$$
	so $L_0 > 0$.
	
	Now, using the polynomial $\mathcal{Z}$-contraction condition for the pair $(r_{m(l)}, r_{n(l)})$, we have
	\[
	\zeta\big(\mathbb{P}(r_{m(l)+1}, r_{n(l)+1}), \mathbb{P}(r_{m(l)}, r_{n(l)})\big) \ge 0.
	\]
	Taking the limit superior as $l \to +\infty$ and applying property $(\zeta_3)$ with \eqref{12p1th5eq} and $L_0 > 0$, we obtain
	\[
	0 \le \limsup_{l\to+\infty} \zeta\big(\mathbb{P}(r_{m(l)+1}, r_{n(l)+1}), \mathbb{P}(r_{m(l)}, r_{n(l)})\big) < 0,
	\]
	which is a contradiction. Therefore, $\{r_n\}$ is Cauchy in $M$. Since $(M,d)$ is complete, therefore there exists  $ r^*\in M$ such that $r_n\stackrel{d}\longrightarrow r^*$.\\
	As $S$ is continuous, we have   
	\begin{equation*}
		r^*=\lim_{n\to +\infty}r_{n+1}=\lim_{n\to +\infty}Sr_n=S(\lim_{n\to +\infty}r_n)=Sr^*,
	\end{equation*}
	this shows $r^*$ is a fixed point of $S$. This completes the proof of existence of fixed point.
	
	Finally, we prove uniqueness. Suppose $u, v \in M$ are two fixed points of $S$ with $u \neq v$. Then $d(u,v) > 0$ and $\mathbb{P}(u,v) > 0$. By the contraction condition and property $(\zeta_2)$,
	\[
	0 \le \zeta\big(\mathbb{P}(S u, S v), \mathbb{P}(u,v)\big) = \zeta\big(\mathbb{P}(u,v), \mathbb{P}(u,v)\big) < \mathbb{P}(u,v) - \mathbb{P}(u,v) = 0,
	\]
	a contradiction. Hence, $u = v$, and the fixed point is unique.	
\end{proof}

\smallskip

\noindent	Next, we present an illustrative example that satisfies the polynomial $\mathcal Z$-contraction condition introduced in this paper, but which is neither a polynomial contraction in the sense of Jleli et al. \cite{Jleli2025} nor a $\mathcal Z$-contraction due to Khojasteh et al. \cite{khojasteh2015new}.
	
	\begin{example}\label{exm1}
	Let $M=[0,+\infty)$ be endowed with the usual metric $d(r,s)=|r-s|$. Then $(M,d)$ is a complete metric space. Define the mapping $S:M\to M$ by
	$$S(r)=\sqrt r,\qquad r\in M.$$
		
	\medskip
	
	\noindent
	\textbf{Claim 1: ($S$ is not a $\mathcal Z$-contraction)}
	Suppose that $S$ is a $\mathcal Z$-contraction. Then there exists a simulation function $\zeta\in\mathcal Z$ such that

	$$\zeta\bigl(d(Sr,Ss),d(r,s)\bigr)\ge0	\quad \text{for all } r,s\in M.$$
	For $r>s>0$, we have
	\begin{equation*}
			d(Sr,Ss)=\sqrt r-\sqrt s=\frac{r-s}{\sqrt r+\sqrt s}.
	\end{equation*}

	Letting $r=s+\varepsilon$ with $\varepsilon\to0^+$, we obtain
	\begin{equation*}
	\frac{d(Sr,Ss)}{d(r,s)}=\frac{1}{\sqrt r+\sqrt s}\longrightarrow+\infty\quad (\text{as}\ r,s\to0^+).
	\end{equation*}

	Hence the distance between images cannot be uniformly controlled by the usual distance, which contradicts the definition of a $\mathcal Z$-contraction. Therefore, $S$ is not a $\mathcal Z$-contraction.
	\medskip
	
	\noindent
	\textbf{Claim 2: ($S$ is not a polynomial contraction)}
	If $S$ satisfies a polynomial contraction, that is,
	\begin{equation*}
	\sum_{i=0}^{k} \alpha_i(Sr,Ss)\, d^i(Sr,Ss)
	\leq
	\lambda
	\sum_{i=0}^{k} \alpha_i(r,s)\, d^i(r,s)
	\end{equation*}

	for some $\lambda\in[0,1)$ and suitable nonnegative coefficient functions $\alpha_i$.
	
	Again, choosing $r=s+\varepsilon$ with $\varepsilon\to0^+$, we have
	\begin{equation*}
	\frac{d(Sr,Ss)}{d(r,s)}\longrightarrow+\infty\quad (\text{as}\ r,s\to0^+),
	\end{equation*}
	which shows that no polynomial expression in $d(Sr,Ss)$ of finite degree can be dominated by a polynomial in $d(r,s)$ with contraction constant $\lambda<1$. Hence $S$ is not a polynomial contraction.
	
	\medskip
	
	\noindent
	\textbf{Claim 3: ($S$ is Polynomial $\mathcal Z$-contraction )}
	Finally, we show that $S$ is a polynomial $\mathcal Z$-contraction for some suitable $\zeta\in\mathcal{Z}$.	Let $k=2$ and define the coefficient functions
	\begin{equation*}
			\alpha_1(r,s)=1,
		\qquad
		\alpha_2(r,s)=\frac{1}{1+r+s},
	\end{equation*}
	which are uniformly continuous on $M\times M$ and satisfy $\alpha_1(r,s)\ge1>0$ for all $r,s\in M$.
	
	Define a simulation function $\zeta:[0,+\infty)\times[0,+\infty)\to\mathbb R$ by
	\begin{equation*}
	\zeta(t,s)=s-t-\frac{s^2}{1+s}.
	\end{equation*}
	Clearly $\zeta\in\mathcal Z$. Then, for $r,s\in M$ with $r\neq s$, we compute
\begin{equation*}
	\begin{aligned}
		\sum_{i=1}^{2} \alpha_i(Sr,Ss)\, d^i(Sr,Ss)
		&=
		(\sqrt r-\sqrt s)
		+\frac{(\sqrt r-\sqrt s)^2}{1+\sqrt r+\sqrt s}.
	\end{aligned}
\end{equation*}

	Using the inequalities
\begin{equation*}
	\sqrt r-\sqrt s < r-s,
	\qquad
	(\sqrt r-\sqrt s)^2\le r-s,
\end{equation*}
	
	we obtain
\begin{equation*}
		\sum_{i=1}^{2} \alpha_i(Sr,Ss)\, d^i(Sr,Ss)
		<
		(r-s)+\frac{(r-s)^2}{1+r-s}.
\end{equation*}
	
	Consequently,
	\begin{equation*}
		\zeta\!\left(
		\sum_{i=1}^{2} \alpha_i(Sr,Ss)\, d^i(Sr,Ss),\;
		\sum_{i=1}^{2} \alpha_i(r,s)\, d^i(r,s)
		\right)\ge0,
	\end{equation*}
	
	i.e.,
	\begin{equation*}
		\zeta(\mathbb{P}(Sr,Ss),\mathbb{P}(r,s))\geq0\ \ \ \text{for all}\ r,s\in M,
	\end{equation*}
	
	\noindent which shows that $S$ is a polynomial $\mathcal Z$-contraction. Also,
	 $S$ is clearly continuous on $M$. Hence, all conditions of Theorem \ref{4p1th} are satisfied and consequently $S$ admits a unique fixed point $r^*\in M$, (namely $r^*=1$). Moreover, for every $r\in M$ the Picard sequence $\{S^n r\}$ converges to $r^*$.
	
	The above example clearly demonstrates that the class of polynomial $\mathcal Z$-contractions strictly contains both the class of polynomial contractions and the class of $\mathcal Z$-contractions, thereby confirming the genuine extension provided in the main result, $i.e.$, Theorem \ref{4p1th}.
	\end{example}

   In the previous theorem we established the existence of a unique fixed point for continuous polynomial $\mathcal{Z}$-contractions. However, in view of Proposition \ref{prop1}, the requirement of continuity is not necessary to hold the result, and therefore can be relaxed. In the following theorem we replaced it by a broader class of mappings, namely Picard continuous mapping that extend our previous main result.

\begin{theorem}\label{4p2th}
	Let $(M,d)$ be a complete metric space, and $S:M\to M$ be a Picard continuous polynomial $\mathcal{Z}$-contraction for some $\zeta\in\mathcal{Z}$. Assume that the coefficient functions $\alpha_i(r,s)$ are uniformly continuous (with respect to the usual product metric) and there exist $j\in \{1,\ldots, k\}$ and a constant $A_j>0$ such that $\alpha_j(r,s)\geq A_j$ for all $r, s\in M$. Then $S$ has a unique fixed point $r^*\in M$ and for every $r\in M$ the sequence $\{S^n{r}\}_{n\in\mathbb{N}}$ converges to $r^*$.
\end{theorem}
\begin{proof}
	On the similar lines of the proof of Theorem \ref{4p1th}, one can show that $\{r_n\}$ is Cauchy sequence. Therefore, due to the completeness of  $(M,d)$, $r_n\stackrel{d}\longrightarrow r^*$ in $M$. Now, as $S$ is Picard continuous, and $\lim_{n\to +\infty}d(r_n,r^*)=0$ (where, $\{r_n\}$ is a Picard sequence generated by $r_0$, i.e., $r_n=S^nr_0$) we must have  $\lim_{n\to +\infty}d(Sr_n,Sr^*)=\lim_{n\to +\infty}d(r_{n+1},Sr^*)=0.$ Then by the uniqueness of limits in metric space, we must have $Sr^*=r^*$. The proof for uniqueness of the fixed point can be established similarly as it is done in Theorem \ref{4p1th}.
\end{proof}
	
The following  example illustrate Theorem \ref{4p2th} for a mapping that is Picard-continuous but not continuous. Consequently, the mapping satisfies the assumptions of Theorem \ref{4p2th} while failing to meet the conditions of Theorem \ref{4p1th} as well as Theorem \ref{khojastehth}. This shows the genuineness of our main result.

\begin{example}\label{exm3} Let $(M,d)=([0,1],|\cdot|)$ and define the mapping $S:M\to M$ by
	\begin{equation*}
			S(r)=
		\begin{cases}
			0, & r\in[0,1),\\[4pt]
			\dfrac12, & r=1.
		\end{cases}
	\end{equation*}
	Clearly $S$ is discontinuous at $r=1$.
%	Let $x_n=1-\frac1n\to1$. Then $S(x_n)=0$ for all $n$, whereas $S(1)=\frac12$.
%	Hence $S(x_n)\to0\neq S(1)$, and $S$ is discontinuous at $x=1$.
	
	To check that $S$ is Picard continuous.	Let $\{r_n\}$ be a Picard sequence defined by $r_{n+1}=S(r_n)$.
	If $r_0\in[0,1)$, then $r_1=0$ and $r_n=0$ for all $n\ge1$.
	If $r_0=1$, then $r_1=\frac12$, $r_2=0$, and again $r_n=0$ for all $n\ge2$.
	Thus every Picard sequence converges to $0$, and
	$$ S(x_n)\to S(0)=0,$$
	which proves that $S$ is Picard continuous.\\
	Here, show that $S$ is in fact a Polynomial $\mathcal{Z}$-contraction.	Let $k=2$ and define the polynomial functional
	$$	P(t)=t+t^2,$$
	corresponding to the constant coefficients $\alpha_0\equiv0$, $\alpha_1\equiv1$ and
	$\alpha_2\equiv1$. Consider the simulation function
	$$	\zeta(t,s)=\frac{s}{1+s}-\frac{t}{1+t}, \qquad t,s\ge0.$$
	
	Clearly, $\zeta\in\mathcal Z$. For arbitrary $r,s\in M$, either $Sr=Ss$, in which case $P(d(Sr,Ss))=0$,
	or $\{r,s\}=\{1,z\}$ with $z<1$, yielding
	$$d(Sr,Ss)=\tfrac12 \le |1-z|=d(r,s).$$
	Hence
	$$P(d(Sr,Ss)) \le P(d(r,s)) \quad  r,s\in M.$$
	Since the function $t\mapsto \frac{t}{1+t}$ is increasing on $[0,+\infty)$,
	it follows that
	$$\zeta\!\left(P(d(Sr,Ss)),P(d(r,s))\right)\ge0,
	\quad  r,s\in M.$$
	Therefore, $S$ is a polynomial $\mathcal Z$-contraction. By Theorem~\ref{4p2th}, $S$ has a unique fixed point $r^*=0$ and for every $r\in M$,	the Picard sequence $\{S^n r\}$ converges to $0$.
\end{example}

The above example demonstrates that Theorem~\ref{4p1th} remains valid for mappings that are \emph{not even continuous}, provided that Picard continuity and the polynomial $\mathcal Z$-contraction condition are satisfied. This shows that the continuity assumption can be replaced by Picard continuity in our main result.

\begin{remark}
	In both of our main results, $i.e.$, Theorem~\ref{4p1th} and \ref{4p2th},
	the global lower bound $\alpha_j(r,s) \ge A_j > 0$ is a sufficient condition adopted for technical simplicity; it is naturally satisfied in our core examples and in many standard settings. Although this assumption can be relaxed by imposing a coercivity condition on $\mathbb{P}$ (e.g., $\mathbb{P}(r,s) \ge \delta(\varepsilon)$ whenever $d(r,s)\ge \varepsilon$), we retain the present form to maintain a transparent proof without extra technicalities.
	
\end{remark}

 A few immediate consequences of our results are listed below as corollaries. 
 \begin{corollary}[Banach contraction principle]
 	Let $(M,d)$ be a complete metric space and $S:M\to M$.
 	Assume there exists $\lambda\in[0,1)$ such that
 	$$	d(Sr,Ss)\le \lambda\, d(r,s),
 	\quad \forall r,s\in M.$$
 	
 	Then $S$ has a unique fixed point $r^*\in M$, and for every $r\in M$,
 	the Picard sequence $\{S^n r\}$ converges to $r^*$.
 	
 	\begin{proof}
 		Let $k=1$ and choose $\alpha_0(r,s)\equiv0$ and $\alpha_1(r,s)\equiv1$ for all $r,s\in M$. Take the simulation function $\zeta(t,s)=\lambda s-t,
 		\quad \lambda\in(0,1).$ Then $\zeta\in\mathcal Z$. Also, $S$ is a polynomial $\mathcal{Z}$-contraction. Moreover, $\alpha_1(r,s)\ge1$, so  Theorem \ref{4p1th}	applies here. Hence $S$ has a unique fixed point
 		and the Picard iteration converges to it.
 	\end{proof}
 	
 \end{corollary}
 
 \begin{corollary}[Polynomial contraction]
 	Let $(M,d)$ be a complete metric space and $S:M\to M$ be continuous.
 	Assume there exist $\lambda\in(0,1)$, $k\ge1$, and uniformly continuous functions
 	$\alpha_i:M\times M\to[0,+\infty)$ such that
 \begin{equation*}
 	\sum_{i=0}^k \alpha_i(Sr,Ss)\, d^i(Sr,Ss)
 	\le
 	\lambda \sum_{i=0}^k \alpha_i(r,s)\, d^i(r,s),
 	\quad \text{for all}\ r,s\in M,
 \end{equation*}
 	and $\alpha_j(r,s)\ge A_j>0$ for some $j\in\{1,\dots,k\}$. Then $S$ has a unique fixed point and Picard iteration converges to the unique fixed point.
 	\begin{proof}
 	Define $\zeta(t,s)=\lambda s-t$, which belongs to $\mathcal Z$.
 		Then
 		\begin{equation*}
  		\zeta\!\left(
 		\sum_{i=0}^k \alpha_i(Sr,Ss)\,d^i(Sr,Ss),\,
 		\sum_{i=0}^k \alpha_i(r,s)\,d^i(r,s)
 		\right)\ge0.
 		\end{equation*}
 		
 		The above inequality is precisely the polynomial $\mathcal{Z}$-contraction
 		condition that required in Theorem \ref{4p1th}. Also, the coefficient functions $\alpha_i$ are uniformly continuous and
 		$\alpha_j(r,s)\ge A_j>0$ for some $j$, all hypotheses of
 		Theorem~\ref{4p1th} hold. Therefore, $S$ admits a unique fixed point
 		and the Picard sequence converges to it.
 	\end{proof}
 	
 \end{corollary}

 \begin{corollary}[$\mathcal Z$-contraction]
 	Let $(M,d)$ be a complete metric space and $S:M\to M$ be continuous.
 	Assume there exists a simulation function $\zeta\in\mathcal Z$ such that
 	\[
 	\zeta\!\left(d(Sr,Ss),\,d(r,s)\right)\ge0,
 	\quad\text{for all}\ r,s\in M.
 	\]
 	Then $S$ admits a unique fixed point in $M$, and the Picard sequence converges
 	to it for every initial point.
 	\begin{proof}
 	This follows from Theorem \ref{4p1th} by choosing $k=1$, $\alpha_0(r,s)\equiv0$ and $\alpha_1(r,s)\equiv1$, which reduces polynomial 
 	$\mathcal{Z}$-contractions	to standard	$\mathcal{Z}$-contractions.
 	\end{proof}
 \end{corollary}
 
\begin{remark}
	The above proofs show that Banach contractions, polynomial contractions,
	and classical $\mathcal Z$-contractions arise as special cases of
	polynomial $\mathcal Z$-contractions by appropriate choices of the
	degree $k$, coefficient functions $\alpha_i$, and the simulation
	function $\zeta$.
\end{remark}
\begin{observation}
 Similarly, from Theorem~\ref{4p2th}, analogous corollaries can also be obtained for Picard continuous mapping  and hence omitted.
\end{observation}

 \section{Application to a Semilinear Elliptic Dirichlet Boundary Value Problem in $L^p(\Omega)$}
    
    In this section, we apply our Theorem~\ref{4p1th} to a semi-linear elliptic boundary value problem formulated in the Lebesgue space $L^p(\Omega)$, thereby extending the applicability of the polynomial $\mathcal Z$-contraction beyond classical sup-norm settings.\\
        
    Let $\Omega\subset \mathbb{R}^n$ $(n\ge1)$ be a bounded domain with smooth boundary $\partial\Omega$. Consider the nonlinear elliptic boundary value problem
    \begin{equation}\label{eq:ellipticLp}
    	\begin{cases}
    		-\Delta u(x) = f(x,u(x)), & x\in \Omega,\\
    		u(x)=0, & x\in \partial\Omega,
    	\end{cases}
    \end{equation}
    where $f:\Omega\times\mathbb{R}\to\mathbb{R}$ is a Carath\'{e}odory function.\\
    
   For $1\le p<+\infty$, Let $M=L^p(\Omega),\ \text{and} \ d(u,v)=\|u-v\|_{L^p(\Omega)}.$ Then $(M,d)$ is a complete metric space.
   Let $G(x,y)$ denote the Green’s function associated with the Laplacian $-\Delta$ on $\Omega$ under homogeneous Dirichlet boundary conditions.
    
    Now, define the operator $S:M\to M$ by
    \begin{equation}\label{eq:operatorLp}
    	(Su)(x)=\int_\Omega G(x,y)\,f(y,u(y))\,dy,
    	\quad x\in\Omega.
    \end{equation}
    It is well known that weak solutions of \eqref{eq:ellipticLp} correspond to fixed points of the operator $S$ in $L^p(\Omega)$. Assume that the Green’s operator is bounded on $L^p(\Omega)$, i.e.,
    \begin{equation}\label{eq:GreenLp}
    	\|Su\|_{L^p(\Omega)} \le C_G \|f(\cdot,u(\cdot))\|_{L^p(\Omega)}
    	\quad \text{for all } u\in L^p(\Omega),
    \end{equation}
    for some constant $C_G>0$.
    
   Also, assume that there exist a natural number $k\ge1$, coefficient functions
    \[
    \alpha_i:M\times M\to[0,+\infty), \quad i=0,1,\ldots,k,
    \]
    and a simulation function $\zeta\in\mathcal Z$ such that for all $r,s\in M$ and \ $x\in\Omega$ a.e.,
    \begin{equation}\label{eq:polyZLp}
    	\zeta\!\left(
    	\sum_{i=0}^k \alpha_i(r,s)\,|f(x,r)-f(x,s)|^i,
    	\sum_{i=0}^k \alpha_i(r,s)\,|r-s|^i
    	\right)\ge 0.
    \end{equation}
    
    Assume further that:
    \begin{itemize}
    	\item each $\alpha_i$ is uniformly continuous w.r.t product $L^p$-norm;
    	\item there exists $j\in\{1,\ldots,k\}$ and a constant $A_j>0$ such that
    	\[
    	\alpha_j(r,s)\ge A_j \quad \text{for all } r,s\in\mathbb{R};
    	\]
    	\item $f(\cdot,u(\cdot))\in L^p(\Omega)$ for all $u\in L^p(\Omega)$.
    \end{itemize}
    
    Under these assumptions, the operator $S$ defined in \eqref{eq:operatorLp} is a continuous polynomial $\mathcal Z$-contraction on $(L^p(\Omega),d)$. Then, by Theorem~\ref{4p1th}, we can ensure that the operator $S$ admits a unique fixed point $u^*\in L^p(\Omega)$. Consequently, problem \eqref{eq:ellipticLp} has a unique weak solution $u^*\in L^p(\Omega)$. Moreover, for any initial function $u_0\in L^p(\Omega)$, the Picard iteration
    $ u_{n+1}=Su_n, \ n\in\mathbb{N}$,
    converges strongly in $L^p(\Omega)$ to the unique solution $u^*$.
	
	\begin{example} Let $\Omega=(0,\pi)$ and let $1\le p<+\infty$. Consider the problem
\begin{equation}\label{BVP}
	\begin{cases}
		-u''(x)	
		=\sin x+\tanh(u(x)-\sin x)
		+\frac{(u(x)-\sin x)^2}{1+(u(x)-\sin x)^2},
		& x\in(0,\pi),\\[6pt]
		u(0)=u(\pi)=0.
	\end{cases}
\end{equation}

Let $M=L^p(\Omega)$ be endowed with the metric $d$ induced by $\|\cdot\|_{L^p(\Omega)}$, then $(M,d)$ is complete.

\medskip

Consider the Green function  $G:\Omega\times\Omega\to\mathbb{R}$  associated with $-u''$ under Dirichlet boundary conditions, given by
\begin{equation}
G(x,y)=
\begin{cases}
	\dfrac{x(\pi-y)}{\pi}, & x\le y,\\[6pt]
	\dfrac{y(\pi-x)}{\pi}, & x>y.
\end{cases}
\end{equation}

Define a operator $S:M\to M$ by
$$(Su)(x):=\int_\Omega G(x,y)\,f(y,u(y))\,dy,$$
where
\begin{equation}\label{eq:fexample}
	f(x,r)=\sin x+\tanh(r-\sin x)
	+\frac{(r-\sin x)^2}{1+(r-\sin x)^2}.
\end{equation}
Then, $f$ is a Carathéodory function since $f(\cdot,r)$ is measurable
for each $r\in\mathbb R$ and $f(x,\cdot)$ is continuous for \ $x\in\Omega$ a.e.
Moreover, the functions $\tanh$ and
$r\mapsto \frac{r^2}{1+r^2}$ are bounded.
Hence,
$$|f(x,u(x))|\le C(1+|u(x)|)\quad\text{a.e.\ in }\Omega,$$
for some constant $C>0$, which implies
$f(\cdot,u(\cdot))\in L^p(\Omega)$ for all $u\in L^p(\Omega)$.

Also, by assumption, the Green operator is bounded on $L^p(\Omega)$, that is,
$$\|Su\|_{L^p(\Omega)}\le C_G\|f(\cdot,u(\cdot))\|_{L^p(\Omega)},
\quad u\in L^p(\Omega),$$ for some constant $C_G>0$.
Now, take
$$\zeta(t,s)=\frac{s}{1+s}-\frac{t}{1+t},
\qquad t,s\ge0,$$
Then, it is easy to check that $\zeta\in\mathcal Z$ and define
$$\alpha_1(r,s)\equiv1,
\qquad
\alpha_2(r,s)\equiv 1,
\quad r,s\in M.$$
For the function
$$f(x,r)=\sin x+\tanh(r-\sin x)
+\frac{(r-\sin x)^2}{1+(r-\sin x)^2}.$$
Since both $\tanh$ and $t\mapsto \frac{t^2}{1+t^2}$ are globally Lipschitz
functions on $\mathbb R$, there exists a constant $L>1$ such that
\begin{equation}\label{eq:Lipschitzf}
	|f(x,r)-f(x,s)|\le L|r-s|,
	\quad
	\text{for all } r,s\in\mathbb R,\ \text{a.e. }x\in\Omega.
\end{equation}

Using \eqref{eq:Lipschitzf}, we obtain
\begin{equation*}
\begin{aligned}
	\sum_{i=1}^2 \alpha_i(r,s)\,|f(x,r)-f(x,s)|^i
	&=
	|f(x,r)-f(x,s)|
	+|f(x,r)-f(x,s)|^2
	\\
	&\le
	L|r-s|+L^2|r-s|^2.
\end{aligned}
\end{equation*}
On the other hand,
$$\sum_{i=1}^2 \alpha_i(r,s)\,|r-s|^i
=|r-s|+|r-s|^2.$$

Setting $t=|r-s|\ge0$, define
$$\Phi(t)=Lt+L^2t^2,\qquad \Psi(t)=t+t^2.$$
Then
$$\sum_{i=1}^2 \alpha_i(r,s)\,|f(x,r)-f(x,s)|^i \le \Phi(t), \qquad
\sum_{i=1}^2 \alpha_i(r,s)\,|r-s|^i = \Psi(t).$$

Next, consider an auxiliary function
$$\phi(u)=\frac{u}{1+u},\qquad u\ge0.$$
The function $\phi$ is strictly increasing on $[0,+\infty)$.
Therefore,
\begin{equation*}
\phi(\Phi(t))=\frac{\Phi(t)}{1+\Phi(t)}\le\frac{\Psi(t)}{1+\Psi(t)}=\phi(\Psi(t)),
\quad t\ge0.
\end{equation*}

Consequently,
\begin{equation*}
\frac{
	\sum_{i=1}^2 \alpha_i(r,s)\,|f(x,r)-f(x,s)|^i
}{
	1+\sum_{i=1}^2 \alpha_i(r,s)\,|f(x,r)-f(x,s)|^i
}
\le
\frac{
	\sum_{i=1}^2 \alpha_i(r,s)\,|r-s|^i
}{
	1+\sum_{i=1}^2 \alpha_i(r,s)\,|r-s|^i
}.
\end{equation*}

Hence,
$$\zeta\!\left(
\sum_{i=1}^2 \alpha_i(r,s)\,|f(x,r)-f(x,s)|^i,\;
\sum_{i=1}^2 \alpha_i(r,s)\,|r-s|^i
\right)\ge0,$$
for all $r,s\in\mathbb R$ and for a.e.\ $x\in\Omega$. Therefore, all the assumptions of Theorem~\ref{4p1th} are satisfied by the boundary value problem \ref{BVP} and the operator $S$ admits a unique fixed point $u^*(x)=\sin x$ and the Picard iteration $u_{n+1}=Su_n$ converges to $u^*$ in $L^p(\Omega)$ for any initial function $u_0\in L^p(\Omega)$(Figure~\ref{fig:PicardConvergence} illustrates the convergence of the Picard iteration). Consequently, $u^*$ is the unique weak solution of \eqref{BVP} in $L^p(\Omega)$.
\begin{figure}[h!]
	\centering
		% Subfigure a: Iterates
	\begin{subfigure}[b]{0.48\textwidth}
		\centering
		\includegraphics[width=\textwidth]{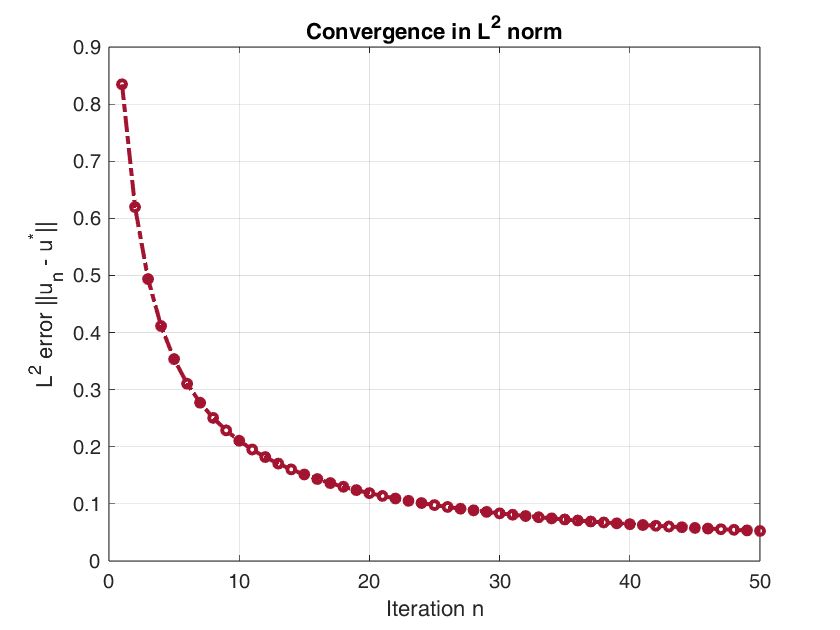} % MATLAB figure 1
		\caption{Decay of $L^2$ error $\|u_n-u^*\|_{L^2(0,\pi)}$ showing convergence.}
		\label{fig:iterates}
	\end{subfigure}
	\hfill
	% Subfigure b: L2 error
	\begin{subfigure}[b]{0.48\textwidth}
		\centering
		\includegraphics[width=\textwidth]{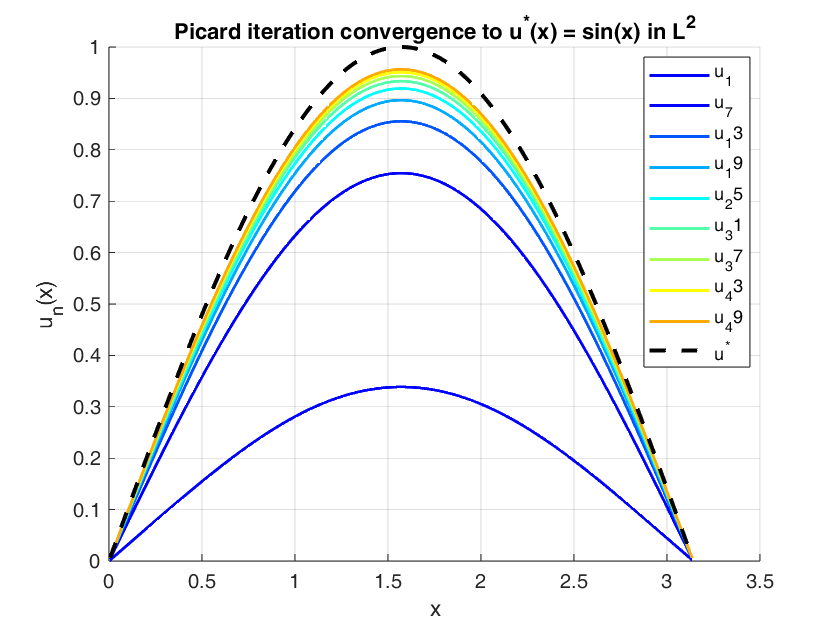} % MATLAB figure 2
		\caption{Picard iterates $u_n(x)$ converging to $u^*(x)=\sin x$ in $L^2(0,\pi)$.}
		\label{fig:L2error}
	\end{subfigure}
	
	\caption{Illustration of Picard iteration under a polynomial $\mathcal Z$-contraction for the BVP \eqref{BVP}.}
	\label{fig:PicardConvergence}
\end{figure}
\end{example}

\noindent {\bf Data availability} The authors confirm that there is no associated data.\\

%\noindent {\bf \large Declarations}\\

\noindent {\bf Funding declaration} No funding.\\

\noindent {\bf Competing interests} The authors declare no competing interests.\\

\noindent{\bf Acknowledgments}  The author wishes to thank Dr. Hayel N. Saleh for his personal communication during the preparation of this manuscript.

%\bibliographystyle{ieeetr}
%\bibliography{bibfile}
\end{document}